\documentclass[leqno]{amsart}
\usepackage{amsmath}
\usepackage{mathtools}
\usepackage{amssymb}
\usepackage{amsthm}
\usepackage{graphicx}
\usepackage{enumerate}
\usepackage[mathscr]{eucal}
\usepackage{tikz}
\usetikzlibrary{decorations.text,calc,arrows.meta}
\theoremstyle{plain}
\newtheorem{theorem}{Theorem}[section]

\newtheorem{cor}{Corollary}[theorem]
\newtheorem{lemma}{Lemma}[section]

\theoremstyle{definition}

\numberwithin{equation}{section}
\usepackage[pagewise]{lineno}

\begin{document}

 \title[Norm derivatives and geometry of bilinear operators]{Norm derivatives and geometry of bilinear operators}
\author[Divya Khurana and  Debmalya Sain]{Divya Khurana and Debmalya Sain}
                \newcommand{\acr}{\newline\indent}
\address[Khurana]{Department of Mathematics\\ Indian Institute of Science\\ Bengaluru 560012\\ Karnataka \\INDIA}
\email{divyakhurana11@gmail.com}

\address[Sain]{Department of Mathematics\\ Indian Institute of Science\\ Bengaluru 560012\\ Karnataka \\INDIA}
\email{saindebmalya@gmail.com}

\thanks{Dr. Sain feels elated to acknowledge his childhood friend Tamal Bandyopadhyay for his inspiring presence in every sphere of his life. The research of Dr. Divya Khurana and Dr. Debmalya Sain is sponsored by Dr. D. S. Kothari Postdoctoral Fellowship under the
mentorship of Professor Gadadhar Misra.}

\subjclass[2010]{Primary 46B20, Secondary 46G25}
\keywords{Birkhoff-James orthogonality; norm derivatives; bilinear
operators; smoothness; norm attainment set}

\begin{abstract}
We study the norm derivatives in the context of Birkhoff-James
orthogonality in real Banach spaces. As an application of this, we
obtain a complete characterization of the left-symmetric points and
the right-symmetric points in a real Banach space in terms of the
norm derivatives. We obtain a complete characterization of strong
Birkhoff-James orthogonality in $\ell_1^n$ and $\ell_\infty^n$
spaces. We also obtain a complete characterization of the
orthogonality relation defined by the norm derivatives in terms of
some newly introduced variation of Birkhoff-James orthogonality. We
further study Birkhoff-James orthogonality, approximate
Birkhoff-James orthogonality, smoothness and norm attainment of
bounded bilinear operators between Banach spaces.
\end{abstract}

\maketitle
\section{Introduction}

The purpose of the present article is to explore the connection
between the norm derivatives and Birkhoff-James orthogonality in
real Banach spaces. We also study smoothness and norm attainment of
bounded bilinear operators between Banach spaces. Let us first
establish the relevant notations and the terminologies to be used in
the present article.

Throughout the text, we use the symbols $X,Y,Z$ to denote real
normed linear spaces. Let $B_X=\{x\in X: \|x\|\leq1\}$ and let
$S_X=\{x\in X: \|x\|=1\}$ be the unit ball and the unit sphere of
$X$, respectively. Let $X^*$ denote the dual space of $X$. Given
$x\in S_X$, $f\in S_{X^*}$ is said to be a supporting functional at
$x$ if $f(x)=\|x\|$. Let $J(x)=\{f\in S_{X^*}: f(x)=\|x\|\}$ denote
the collection of all supporting functionals at $x$. Note that for
each $x\in S_X,$ the Hahn-Banach theorem ensures the existence of at
least one supporting functional at $x$.

If $x,y\in X,$ then we say that $x$ is Birkhoff-James orthogonal to
$y$, written as $x\perp_B y$, if $\|x+\lambda y\|\geq \|x\|$ for all
$\lambda\in \mathbb{R}$. We refer the readers to the pioneering
articles \cite{B, J1} for more information in this regard. In
\cite[Theorem 2.1]{J1}, James proved that $x\perp_B y$ if and only
if there exists $f\in J(x)$ such that $f(y)=0$. Given $0\not=x,y\in
X,$  $x$ is said to be strongly orthogonal to $y$ in the sense of
Birkhoff-James \cite{PSJ}, written as $x\perp_{SB}y,$ if
$\|x+\lambda y\|>\|x\|$ for all $\lambda\not=0$. For $x,y \in X$ and
$\epsilon\in [0,1)$, $x$ is said to be approximate
$\epsilon$-orthogonal to $y$ \cite{C}, written as $x\perp_B^\epsilon
y$, if $\|x+\lambda y\|^2\geq \|x\|^2-2\epsilon\|x\|\|\lambda y\|$
for all $\lambda \in \mathbb{R}$. Observe that Birkhoff-James
orthogonality is homogeneous.

Sain \cite{S} characterized Birkhoff-James orthogonality of linear
operators on finite-dimensional Banach spaces by introducing the
notions of the positive part of $x$, denoted by $x^+$, and the
negative part of $x$, denoted by $x^-$, for an element $x\in X$. For
any element $y\in X$, we say that $y\in x^+$ $(y\in x^-)$ if
$\|x+\lambda y\|\geq \|x\|$ for all $\lambda \geq 0$ $(\lambda \leq
0)$. It is easy to see that $x^\perp=\{y\in X: x\perp_B y\}=x^+\cap
x^-$. In general, Birkhoff-James orthogonality relation between two
elements need not be symmetric. In other words, for any two elements
$x,y\in X$, $x\perp_B y$ does not necessarily imply $y\perp_Bx$. An
element $x\in X$ is said to be left-symmetric (right-symmetric)
\cite{S} if $x\perp_B y$ ($y\perp_B x$) implies $y\perp_B x$
($x\perp_B y$). James \cite{J} proved that if dim $X\geq 3$ and
Birkhoff-James orthogonality is symmetric then the norm is induced
by an inner product.

An element $x\in S_X$ is said to be smooth point if $J(x)=\{ f\}$
for some $f\in S_{X^*}$. A Banach space $X$ is said to be smooth if
every $x\in S_X$ is a smooth point. The characterization of smooth
points obtained by James \cite{J1} has been used in our study, which
states that $0\not=x \in X$ is a smooth point in $X$ if and only if
for any $y, z \in X, $ $x\perp_B y$ and $x\perp_B z$ implies that
$x\perp_B (y+z)$. It is well known that smoothness of $x\in S_X$ is
equivalent to the G\^{a}teaux differentiability of norm at $x$,
$i.e.$, $x\in S_X$ is a smooth point in $X$ if and only if
$\rho(x,y)=\lim_{\lambda\rightarrow 0}\frac{\|x+\lambda
y\|-\|x\|}{\lambda}$ exists for all $y\in X$. We recall the
following standard notions of left-hand and right-hand G\^{a}teaux
derivative of norm at $x\in X$ in direction of $y\in X:$
\begin{align*}
\rho_{+}(x,y)=\lim_{\lambda\rightarrow 0^{+}}\frac{\|x+\lambda
y\|-\|x\|}{\lambda}, ~~\rho_{-}(x,y)=\lim_{\lambda\rightarrow
0^{-}}\frac{\|x+\lambda y\|-\|x\|}{\lambda}.
\end{align*}

We will use the following properties of norm derivatives in this
note (see \cite{AST} and \cite{M} for details):

\begin{itemize}
\item[(i)]For all $x$, $y\in X$ and all $\alpha\in\mathbb{R}$,
\[\rho_\pm(\alpha x, y)=\rho_\pm(x,\alpha y) = \left\{
    \begin{array}{ll}
        \alpha \rho_\pm(x,y)  & \mbox{if } \alpha\geq 0 \\
       \alpha \rho_\mp(x,y)  & \mbox{if } \alpha< 0.
    \end{array}
\right.\]

\item[(ii)] $\rho_{+}(x,y)=\rho_{-}(x,y)$ for all $y\in X$ if
and only if $0 \not=x$ is a smooth point in $X$,

\item[(iii)] $\rho_{\pm}$ are continuous with respect to the second
variable.
\end{itemize}

We note that there are several notions of orthogonality in a normed
linear space which are equivalent only if the norm is induced by an
inner product \cite{AMW}. Indeed, the above mentioned concepts of $
\rho_{\pm} $ yields the following notions of orthogonality in normed
linear spaces \cite{CW1}. Given $ x, y \in X, $

\[ x \perp_{\rho_{+}} y \Longleftrightarrow \rho_{+} (x, y) = 0; \]
 \[ x \perp_{\rho_{-}} y \Longleftrightarrow \rho_{-} (x, y) = 0; \]
 \[x \perp_{\rho} y \Longleftrightarrow \rho_{+} (x, y) + \rho_{-} (x, y) = 0.\]

As mentioned in \cite{CW1}, the relations $ \perp_{\rho_{+}},
\perp_{\rho_{-}} $ and $ \perp_{\rho} $ are equivalent in an inner
product space but are generally incomparable in a normed space which
is not smooth. Nevertheless, we illustrate that it is possible to
establish a connection between $ \perp_{\rho_{+}}, \perp_{\rho_{-}}
$ and Birkhoff-James orthogonality. To serve our purpose, we
introduce the following variation of Birkhoff-James orthogonality,
denoted by $ \perp_{B^*}.$ If $x,y\in S_X$ then we write
$x\perp_{B^*}y$ if $x\perp_B y$ and $x {\not\perp}_B u$ for any $u$
of the form $u=ty+(1-t)x,$ where $t\in (0,1)$.

If $X$ and $Y$ are normed linear (Banach/reflexive Banach) spaces
then it is easy to see that $X\times Y$ is a normed linear
(Banach/reflexive Banach) space equipped with the norm
$\|(x,y)\|_\infty=\max \{\|x\|,\|y\|\}$ for all $(x,y)\in X\times
Y$.  Throughout this article we will consider $X\times Y$ with the
aforesaid norm. Let $\mathcal{T}:X\times Y\rightarrow Z$ be a
bilinear operator, $i.e.$, $\mathcal{T}$ is linear in each argument.
The norm of $\mathcal{T}$ is defined as $\|\mathcal{T}\|=\sup
\{\|\mathcal{T}(x,y)\|:(x,y)\in S_{X} \times S_{Y}\}$. Let
$M_{\mathcal{T}}$ denote the norm attainment set of $\mathcal{T}$,
$i.e.$, $M_{\mathcal{T}}=\{(x,y)\in S_{X} \times
S_{Y}:\|{\mathcal{T}}(x,y)\|=\|{\mathcal{T}}\|\}$.
$\{(x_n,y_n)\}\subset S_{X} \times S_{Y}$ is said to be a norming
sequence for $\mathcal{T}$ if
$\|\mathcal{T}(x_n,y_n)\|\rightarrow\|\mathcal{T}\|$. We say that
$\mathcal{T}$ is bounded if $\|\mathcal{T}\|<\infty$ and
$\mathcal{T}$ is compact if for all bounded sequences
$\{(x_n,y_n)\}\subset X\times Y$, the sequence
$\{\mathcal{T}(x_n,y_n)\}$ has a convergent subsequence. Let
$\mathscr{B}(X\times Y,Z)$ $(\mathscr{K}(X\times Y,Z))$ denote the
space of all bounded (compact) bilinear operators from $X\times
Y$ to $Z$.\\
The concept semi-inner-product (s.i.p.) in normed linear spaces
\cite{L} plays an important role in our discussion of the geometry
of bilinear operators. Giles \cite{G} proved that for every normed
linear space $X$, there exists a s.i.p. on $ X $ which is compatible
with the norm, $i.e.$, $[x,x]=\|x\|^2$ for all $x\in X$. Also, for a
normed linear space $X$, s.i.p. on $ X $ which is compatible with
the norm is unique if and only if $X$ is smooth. For $X=\ell_p^n$,
$1<p<\infty$, the following is the unique s.i.p. on $X$ compatible
with its norm:
\begin{align*}
[x,y]=\frac{1}{\|y\|^{p-2}}\sum_{i=1}^nx_iy_i|y_i|^{p-2}
\end{align*}
where $x=(x_1,\ldots,x_n),~y=(y_1,\ldots, y_n)\in
\ell_p^n\setminus\{0\}$.

Whenever we work with a s.i.p. in a normed linear space, we assume
that the s.i.p. is compatible with the norm. We will use the
following result related to s.i.p. and Birkhoff-James orthogonality
in normed linear spaces.

\begin{theorem}\cite{G}\label{sip and BJ}
Let $X$ be a normed linear space and let $x$, $y\in X$. If $x\perp_B
y$ then there exists a s.i.p. $[~,~]$ on $X$ such that $[y, x]=0$.
\end{theorem}


After this introductory part, this article is demarcated into two
sections. The first section deals with the norm derivatives and its
connections with geometry of the space and Birkhoff-James
orthogonality. In the next section, we explore Birthkhoff-James
orthogonality, approximate Birthkhoff-James orthogonality,
smoothness and the norm attainment of bounded bilinear operators
between Banach spaces. These results extend some of the recent
observations regarding the analogous properties of bounded bilinear
operators between Banach spaces \cite{DS}.


\section{norm derivatives}
We start with the following complete characterizations of  the
positive part and  the negative part of an element in a normed
linear space, in terms of the norm derivatives $\rho_{\pm}$ and
s.i.p. on the given normed linear space.
\begin{theorem}\label{positive}
Let $X$ be a normed linear space and let $x, y\in X$. Then the
following are equivalent:
\begin{itemize}
\item[(a)] $\rho_+(x,y) \geq 0$.
\item[(b)] $y\in x^+$.
\item[(c)] There exists a s.i.p. $[~,~]$ on $X$ such that $[y,x]\geq
0$.
\end{itemize}
\end{theorem}

\begin{proof}
We first prove the equivalence of $(a)$ and $(b)$. To prove that
$(a)$ implies $(b)$ we consider the following two
cases.\\

Case I: Let $\rho_+(x,y)>0$. Then there exists sufficiently small
$\lambda_0>0$ such that $\|x+\lambda y\|-\|x\|>0$ for all
$\lambda\in(0,\lambda_0]$. Now, by the convexity of the norm it
follows that $\|x+\lambda y\|\geq\|x\|$ for all $\lambda\geq 0$ and
hence $u\in x^+$.\\

Case II: Let $\rho_+(x,y)=0$. Then it follows from \cite[Theorem
3.2]{CW} that $x\perp_B y$ and thus $y\in x^+$.

Now, we show that $(b)$ implies $(a)$. Observe that $(b)$ implies,
$\|x+\lambda y\|\geq \|x\|$ for all $\lambda\geq 0$. Thus
$\rho_{+}(x,y)=\lim_{\lambda\rightarrow 0^{+}}\frac{\|x+\lambda
y\|-\|x\|}{\lambda}\geq 0$ and $(a)$ follows.

Equivalence of $(b)$ and $(c)$ follows from \cite[Theorem 2.4]{SPM}.
\end{proof}

The following result follows from similar arguments as used in the
above theorem.

\begin{theorem}\label{negative}
Let $X$ be a normed linear space and let $x, y\in X$. Then the
following are equivalent:
\begin{itemize}
\item[(a)] $\rho_-(x,y) \leq 0$.
\item[(b)] $y\in x^-$.
\item[(c)] There exists a s.i.p. $[~,~]$ on $X$ such that $[y,x]\leq
0$.
\end{itemize}
\end{theorem}

In \cite{S2}, Sain et al. obtained a complete characterization of
the left-symmetric points and the right-symmetric points in a normed
linear space in terms of the positive part and the negative part of
the elements of the space. As a consequence of
Theorem~\ref{positive}, Theorem~\ref{negative} and results obtained
in \cite{S2}, we obtain the following complete characterizations of
the left-symmetric points and the right-symmetric points of a normed
linear space in terms of the norm derivatives $\rho_{\pm}$.

\begin{cor}
Let $X$ be a normed linear space and let $x\in X$. Then the
following are equivalent:
\begin{itemize}
\item[(a)] $x$ is a left-symmetric point in $X$.
\item[(b)] Given $y\in X$ with $\rho_+(x,y) \geq 0$ implies that $\rho_+(y,x) \geq 0$.
\item[(c)] Given $y\in X$ with $\rho_-(x,y) \leq 0$ implies that $\rho_-(y,x) \leq 0$.
\end{itemize}
\end{cor}
\begin{proof}
Let $(a)$ holds true and let $y\in X$ be such that $\rho_+(x,y) \geq
0$. Then it follows from Theorem~\ref{positive} that $y\in x^+$.
Now, it follows from \cite[Theorem 2.1]{S2} that $x\in y^+$. Again,
using Theorem~\ref{positive} it follows that $\rho_+(y,x) \geq 0$
and thus $(b)$ follows.

Let $(b)$ holds true and let $y\in X$ with $\rho_-(x,y) \leq 0$.
Then by using the properties of the norm derivatives we get
$\rho_+(x,-y)\geq 0$. Now, $(b)$ implies $\rho_+(-y, x)\geq 0$.
Again, using the properties of the norm derivatives we get
$\rho_-(y,x) \leq 0$. Thus $(c)$ follows.

Let $(c)$  holds true. To prove that $(a)$ holds true \cite[Theorem
2.1]{S2} implies, it is sufficient to show that given  $y\in X$ such
that $y\in x^-$ implies $x\in y^-$. Let $y\in X$ be such that $y\in
x^-$, then it follows from Theorem~\ref{negative} that $\rho_-(x,y)
\leq 0$. Now, $(c)$ implies that $\rho_-(y,x) \leq 0$ and again
using Theorem~\ref{negative} it follows that $x\in y^-$. Thus $(a)$
follows from \cite[Theorem 2.1]{S2}.
\end{proof}

Using similar arguments as in the above corollary and \cite[Theorem
2.2]{S2} the following result follows.

\begin{cor}
Let $X$ be a normed linear space let $x\in X$. Then the following
are equivalent:
\begin{itemize}
\item[(a)] $x$ is a right-symmetric point in $X$.
\item[(b)] Given $y\in X$ with $\rho_+(y,x) \geq 0$ implies that $\rho_+(x,y) \geq 0$.
\item[(c)] Given $y\in X$ with $\rho_-(y,x) \leq 0$ implies that $\rho_-(x,y) \leq 0$.
\end{itemize}
\end{cor}

We would like to give an explicit description of the sign of the
norm derivatives in $\ell_p^n$ spaces, $1\leq p\leq \infty$. Note
that when $1<p<\infty$, the corresponding space is smooth and
therefore, the desired description follows by using the unique
s.i.p. expression in the said space.


\begin{theorem}
Let $x=(x_1,\ldots,x_n)$, $y=(y_1,\ldots,y_n)\in
{\ell_p^n}\setminus\{0\}$, $1<p<\infty$. Then
\begin{itemize}
\item[(a)] $\rho_+(x,y)\geq 0$ if and only if $\sum\limits_{i=1}^ny_ix_i|x_i|^{p-2}\geq 0$,
\item[(b)] $\rho_-(x,y)\leq 0$ if and only if $\sum\limits_{i=1}^ny_ix_i|x_i|^{p-2}\leq
0$.
\end{itemize}
\end{theorem}

\begin{proof}
If $x=(x_1,\ldots,x_n)$, $y=(y_1,\ldots,y_n)\in
{\ell_p^n}\setminus\{0\}$, $1<p<\infty$, then the unique s.i.p. on
${\ell_p^n}$, $1<p<\infty$, is given by
\begin{align*}
[y,x]=\frac{1}{\|x\|^{p-2}}\sum_{i=1}^ny_ix_i|x_i|^{p-2}.
\end{align*}
Now, the proof of $(a)$, $(b)$ follows from Theorem~\ref{positive}
and Theorem~\ref{negative}, respectively.
\end{proof}

When $p=1,\infty$, the corresponding $\ell_p^n$ space is not smooth
and there exist infinitely many s.i.p. on the space. Our next result
addresses this issue by means of the direct computation. If $t\in
\mathbb{R}$ then $sgn~t$ denotes the sign function, $i.e.$,
$sgn~t=\frac{t}{|t|}$ for $t\not=0$ and $sgn~0=0$. We first write a
simple observation: if $x, y\in \mathbb{R}$, where $x\not=0$, then
there exists sufficiently small $\lambda_0>0$ such that $|x+\lambda
y|=|x|+\lambda (sgn~x)y$ for all $\lambda\in
[-\lambda_0,\lambda_0]$.

\begin{theorem}
Let $x=(x_1,\ldots,x_n),y=(y_1,\ldots,y_n)\in \ell_\infty^n$. Then
\begin{itemize}
\item[(a)] $\rho_+(x,y)\geq 0$ if and only if there exists $1\leq i_0\leq n$ such that $\|x\|=|x_{i_0}|$ and $(sgn~ x_{i_0}) y_{i_0} \geq
0$.
\item[(b)] $\rho_-(x,y)\leq 0$ if and only if there exists $1\leq i_0\leq n$ such that $\|x\|=|x_{i_0}|$ and $(sgn~ x_{i_0}) y_{i_0} \leq
0$.
\end{itemize}
\end{theorem}

\begin{proof}
$(a)$ If $x=0$ then the result follows trivially. Now, we assume
that $x\not=0$. We first prove the sufficient part. It follows from
Theorem~\ref{positive} that to prove the sufficient part it is
enough to show that $y\in x^+$.  Let $1\leq i_0\leq n$ be such that
$\|x\|=|x_{i_0}|$, $(sgn~ x_{i_0}) y_{i_0} \geq 0$. Then for any
$\lambda\geq 0$, we have,
\begin{align*}
\|x+\lambda y\|&=\max_{1\leq i \leq n}|x_i+\lambda y_i|\\&\geq
|x_{i_0}+\lambda y_{i_0}|\\&=||x_{i_0}|+\lambda (sgn~ x_{i_0})
y_{i_0}|\\&\geq |x_{i_0}|+\lambda (sgn~ x_{i_0}) y_{i_0}\\&\geq
|x_{i_0}|\\&=\|x\|.
\end{align*}
Thus for all $\lambda\geq 0$, $\|x+\lambda y\|\geq \|x\|$ and hence
the sufficient part follows.

We now prove the necessary part. Let $\rho_+(x,y)\geq 0$, then it
follows from Theorem~\ref{positive} that $y\in x^+$. Let $\{
j_1,\ldots,j_m\}\subseteq \{1,\ldots,n\}$, $m\leq n$, be the maximal
subset such that $\|x\|=|x_{j_k}|$ for all $1\leq k \leq m$. Suppose
on the contrary that $(sgn~x_{j_k}) y_{j_k}<0$ for all $1\leq k\leq
m$. We can choose sufficiently small $\lambda_1>0$ and $1\leq l\leq
m$ such that for all $\lambda\in (0,\lambda_1]$, we have,
\begin{align*}
\|x+\lambda y\|=\max_{1\leq i\leq n}|x_i+\lambda
y_i|=|x_{j_l}|+\lambda (sgn~x_{j_l}) y_{j_l}<|x_{j_l}|=\|x\|.
\end{align*}

This leads to a contradiction and thus the necessary part follows.

Proof of $(b)$ follows from similar arguments as above.
\end{proof}

\begin{theorem}
Let $x=(x_1,\ldots,x_n),y=(y_1,\ldots,y_n)\in \ell_1^n$. Then
$\rho_+(x,y)\geq 0$ if and only either of the following hold true:
\begin{itemize}
\item[(a)] $\{i:1\leq i\leq n,~ x_i\not=0~\mbox{and}~y_i\not=0\}=\emptyset$.\\
\item[(b)] $\{i:1\leq i\leq n,~
x_i\not=0~\mbox{and}~y_i\not=0\}\not=\emptyset$ and
\begin{align*}
\sum_{\substack{i=1\\x_i\not=0}}^n (sgn~ x_i) y_i+
\sum_{\substack{i=1\\x_i=0}}^n |y_i| \geq 0.
\end{align*}
\end{itemize}
\end{theorem}

\begin{proof}
For the necessary part  assume that $\rho_+(x,y)\geq 0$. Now, from
Theorem~\ref{positive}, it follows that $y\in x^+$. If $x$, $y$ are
such that $\{i:1\leq i\leq n,~
x_i\not=0~\mbox{and}~y_i\not=0\}=\emptyset$ then $(a)$ follows. Now,
let $\{i:1\leq i\leq n,~
x_i\not=0~\mbox{and}~y_i\not=0\}\not=\emptyset$. Suppose on the
contrary that $\sum\limits_{\substack{i=1\\x_i\not=0}}^n (sgn~ x_i)
y_i+ \sum\limits_{\substack{i=1\\x_i=0}}^n |y_i| < 0$.

We can find sufficiently small $\lambda_1>0$ such that for all
$\lambda \in (0,\lambda_1]$, we have,
\begin{align*}
\|x+\lambda y\|&=\sum_{\substack{i=1\\x_i\not=0}}^n |x_i+\lambda
               y_i|+\lambda
               \sum_{\substack{i=1\\x_i=0}}^n |y_i|\\
               &=\sum_{\substack{i=1\\x_i\not=0}}^n ||x_i|+\lambda(sgn~x_i)
               y_i|+\lambda
               \sum_{\substack{i=1\\x_i=0}}^n |y_i|\\
               &=\sum_{\substack{i=1\\x_i\not=0}}^n (|x_i|+\lambda(sgn~x_i)
               y_i)+\lambda
               \sum_{\substack{i=1\\x_i=0}}^n |y_i|\\
               &=\|x\|+\lambda (\sum_{\substack{i=1\\x_i\not=0}}^n (sgn~ x_i) y_i+
               \sum_{\substack{i=1\\x_i=0}}^n |y_i|)\\
               &<\|x\|.
\end{align*}
This contradicts that $y\in x^+$. Thus $(b)$ follows.

Now we prove the sufficient part. Let $(a)$ holds true. Then it
follows easily that $x\perp_B y$, $i.e.$, $\|x+\lambda y\|\geq
\|y\|$ for all $\lambda\in \mathbb{R}$. Now, it follows from the
definition of $\rho_+(x,y)$ that $\rho_+(x,y)\geq 0$.  We now assume
that $(b)$ holds true. Observe that from Theorem~\ref{positive}, it
is sufficient to show that $y\in x^+$. If $\lambda\geq 0$ then
\begin{align*}
\|x+\lambda y\|&=\sum_{i=1}^n|x_i+\lambda y_i| \\
               &=\sum_{\substack{i=1\\x_i\not=0}}^n |x_i+\lambda
               y_i|+\lambda
               \sum_{\substack{i=1\\x_i=0}}^n |y_i|\\
               &=\sum_{\substack{i=1\\x_i\not=0}}^n ||x_i|+\lambda(sgn~x_i)
               y_i|+\lambda
               \sum_{\substack{i=1\\x_i=0}}^n |y_i|\\
               &\geq  \sum_{\substack{i=1\\x_i\not=0}}^n (|x_i|+\lambda(sgn~x_i)
               y_i)+\lambda
               \sum_{\substack{i=1\\x_i=0}}^n |y_i|\\
               &=\|x\|+ \lambda (\sum_{\substack{i=1\\x_i\not=0}}^n (sgn~x_i) y_i+
               \sum_{\substack{i=1\\x_i=0}}^n |y_i|)\\
               &\geq\|x\|.
\end{align*}
Thus $y\in x^+$ and $\rho_+(x,y)\geq 0$.
\end{proof}

Similar arguments yields the following analogous characterization of
$\rho_-(x,y)\leq 0$ for $x,y\in\ell_1^n$.

\begin{theorem}
Let $x=(x_1,\ldots,x_n),y=(y_1,\ldots,y_n)\in \ell_1^n$. Then
$\rho_-(x,y)\leq 0$ if and only either of the following hold true:
\begin{itemize}
\item[(a)] $\{i:1\leq i\leq n,~ x_i\not=0~\mbox{and}~y_i\not=0\}=\emptyset$.\\
\item[(b)] $\{i:1\leq i\leq n,~
x_i\not=0~\mbox{and}~y_i\not=0\}\not=\emptyset$ and
\begin{align*}
\sum_{\substack{i=1\\x_i\not=0}}^n (sgn~ x_i) y_i+
\sum_{\substack{i=1\\x_i=0}}^n |y_i| \leq 0.
\end{align*}
\end{itemize}
\end{theorem}

In \cite{PSJ}, Paul et al. observed that a normed linear space $X$
is strictly convex if and only if $x,y\in S_X$ and $x\perp_B y$
implies that $x\perp_{SB} y$. Thus for $\ell_p^n$ spaces, where
$1<p<\infty$, Birkhoff-James orthogonality and strong
Birkhoff-James orthogonality coincide. For $\ell^n_1$ or
$\ell^n_{\infty}$, characterization of strong Birkhoff-James
orthogonality is not known. Our next result provides a complete
characterization of strong Birkhoff-James orthogonality, involving
the norm derivatives $\rho_{\pm}$ in $\ell_1^n$ and
$\ell^n_{\infty}$ spaces.

\begin{theorem}\label{strong BJ}
Let $X=\ell^n_1$ or $\ell^n_{\infty}$, and let $x,y\in S_X$. Then
$x\perp_{SB}y$ if and only if the following two conditions hold
true:

\begin{itemize}
\item[(a)] $\rho_+(x,y)>0$,
\item[(b)] $\rho_-(x,y)<0$.
\end{itemize}
\end{theorem}

\begin{proof}
Let us first prove the necessary part of the theorem. Let
$x\perp_{SB}y$. Then $y\in x^+$ and $y\in x^-$. Now, it follows from
Theorem~\ref{positive} and Theorem~\ref{negative} that
$\rho_+(x,y)\geq 0$, $\rho_-(x,y)\leq 0$. Suppose on the contrary
that $\rho_+(x,y)=0$.

If $X=\ell_1^n$ then we can find a sufficiently small $\lambda_1>0$
such that for all $\lambda \in (0,\lambda_1]$, we have,
\begin{align*}
\|x+\lambda y\|&=\sum_{\substack{i=1\\x_i\not=0}}^n |x_i+\lambda
               y_i|+\lambda
               \sum_{\substack{i=1\\x_i=0}}^n |y_i|\\
               &=\sum_{\substack{i=1\\x_i\not=0}}^n ||x_i|+\lambda(sgn~x_i)
               y_i|+\lambda
               \sum_{\substack{i=1\\x_i=0}}^n |y_i|\\
               &=\sum_{\substack{i=1\\x_i\not=0}}^n (|x_i|+\lambda(sgn~x_i)
               y_i)+\lambda
               \sum_{\substack{i=1\\x_i=0}}^n |y_i|\\
               &=\|x\|+\lambda (\sum_{\substack{i=1\\x_i\not=0}}^n (sgn~ x_i) y_i+
               \sum_{\substack{i=1\\x_i=0}}^n |y_i|).
\end{align*}

If $X=\ell_\infty^n$ then we can find a sufficiently small
$\lambda_2>0$, $1\leq i_0\leq n$, such that $\|x\|=|x_{i_0}|$ and
for all $\lambda \in (0,\lambda_1]$, we have,
\begin{align*}
\|x+\lambda y\|&=\max_{1\leq i \leq n}|x_i+\lambda
y_i|\\&=\max_{1\leq i \leq n}||x_i|+\lambda (sgn
~x_i)y_i|\\&=|x_{i_0}|+\lambda (sgn~ x_{i_0})
y_{i_0}\\&=\|x\|+\lambda (sgn~ x_{i_0}) y_{i_0}.
\end{align*}

Thus in both cases we can find a sufficiently small $\alpha>0$ such
that for all $\lambda\in (0,\alpha]$, we have,
\begin{align*}
\|x+\lambda y\|=\|x\|+\lambda a,
\end{align*}
where
\[a = \left\{
    \begin{array}{ll}
        \sum\limits_{\substack{i=1\\x_i\not=0}}^n (sgn~ x_i) y_i+
               \sum\limits_{\substack{i=1\\x_i=0}}^n |y_i|  & \mbox{if } x \in\ell_1^n \\
        (sgn~
x_{i_0}) y_{i_0}\mbox{~for~some }1\leq i_0\leq n &\mbox{if }
x\in\ell_\infty^n.
    \end{array}
\right.\]

Now, the assumption $\rho_+(x,y)=0$ implies that $a=0$ and thus
$\|x+\lambda y\|-\|x\|=0$ for all $\lambda \in (0,\alpha]$. This
contradicts the assumption that $x\perp_{SB}y$. Thus
$\rho_+(x,y)>0$. Using similar arguments we can show that
$\rho_-(x,y)<0$.

Let us now prove the sufficient part of the theorem. If
$\rho_+(x,y)>0$ then we can find a sufficiently small $\lambda_3>0$
such that for all $\lambda \in (0,\lambda_3]$, we have,
\begin{equation*}
\|x+\lambda y\|-\|x\|>0.
\end{equation*}

Now, by the convexity of the norm we obtain that
\begin{equation*}
\|x+\lambda y\|>\|x\|
\end{equation*}

for all $\lambda>0$. Similar arguments as above show that if
$\rho_-(x,y)<0$ then we can find a sufficiently small $\lambda_4<0$
such that for all $\lambda \in [\lambda_4,0)$, we have,
\begin{equation*}
\|x+\lambda y\|-\|x\|>0.
\end{equation*}
Now, again by the convexity of the norm we obtain that
\begin{equation*}
\|x+\lambda y\|>\|x\|
\end{equation*}
for all $\lambda<0$. Thus $\|x+\lambda y\|>\|x\|$ for all
$\lambda\not=0.$ This proves the sufficient part of the theorem.
\end{proof}

As a corollary to the above result, we obtain the following complete
characterization of strong Birkhoff-orthogonality in $X\oplus_1Y$,
where $X=\ell_1^n$ or $\ell^n_{\infty}$,  $Y=\ell_1^m$ or
$\ell^m_{\infty}$ and $X\oplus_1Y=\{(x,y):x\in X,y\in Y\}$ is a
Banach space with the norm $\|(x,y)\|=\|x\|+\|y\|$.

\begin{cor}
Let $X=\ell^n_1$ or $\ell^n_{\infty}$, and let $Y=\ell_1^m$ or
$\ell^m_{\infty}.$ Let $(x,y),(u,v)\in X\oplus_1 Y$. Then
$(x,y)\perp_{SB}(u,v)$ if and only if the following two conditions
hold true:

\begin{itemize}
\item[(a)] $\rho_+(x,u)+ \rho_+(y,v)>0$,
\item[(b)] $\rho_-(x,u)+\rho_-(y,v)<0$.
\end{itemize}
\end{cor}

We would like to conclude this section by obtaining  complete
characterizations of $\rho_+$ orthogonality and $\rho_-$
orthogonality in normed linear space.  First we recall the following
result from \cite{M}, presenting a relationship between functionals
supporting $B_X$ at $x$ and $\rho_{\pm}(x,.)$.

\begin{lemma}\cite[Lemma 5.4.16]{M} \label{normderivative}
Let $X$ be a normed linear space and let $x\in S_X$. Then $f\in
S_{X^*}$ is a supporting functional at  $x$ if and only if for any
$y\in X,$
\begin{align*}
\rho_{-}(x,y)\leq f(y) \leq \rho_+(x,y).
\end{align*}
\end{lemma}

For our next result characterizing $\rho_+$ and $\rho_-$
orthogonality, we recall the definition of a normal cone in normed
linear spaces. A subset $K$ of $ X $ is said to be a normal cone if

\noindent $(1)~ K + K \subset K,$\\
$(2)~ \alpha K \subset K $ for all $ \alpha \geq 0 $ and\\
$(3)~ K \cap (-K) = \{0\}. $

\begin{theorem}
Let $X$ be a Banach space and let $x,y\in S_X$. Then the following
results hold true:

\begin{itemize}
\item[(a)] $\rho_+(x,y)=0$ if and only if $(-x)\perp_{B^*}y$.

\item[(b)] $\rho_-(x,y)=0$ if and only if $x\perp_{B^*}y$.
\end{itemize}
\end{theorem}

\begin{proof}
Let us first prove the necessary part of $(a)$. Since
$\rho_+(x,y)=0$, by \cite[Theorem 3.2]{CW}, it follows that
$x\perp_B y$. If $x$ is a smooth point then $x^{\perp}\cap
span\{x,y\}=\{\alpha y:\alpha \in \mathbb{R}\}$. This situation is
illustrated in figure 1 below. In this case, $(-x)\perp_{B^*}y$ is
trivially true. Now, assume that $x$ is a non-smooth point. If
$H=x^{\perp}\cap span\{x,y\}$ then by taking $\epsilon=0$ in
\cite[Theorem 2.1]{SPM1}, it follows that $H=K\cup (-K)$, where $K$
is a normal cone in $span\{x,y\}$. Let $K$ be generated by $v_1$ and
$v_2$, $i.e.$, $K=\{\alpha v_1+\beta v_2: \alpha\geq0, \beta\geq
0\}$. This situation is illustrated in figure 2 below. We now prove
that if $y\in K$ then $y=v_2$ and similar arguments will show that
if $y\in-K$ then $y=-v_1$. Let $y\in K$ and $y\not=v_2$. Then
$\mathcal{A}=\{z\in S_X\cap K:
z=\frac{tv_2+(1-t)y}{\|tv_2+(1-t)y\|}, t\in(0,1)\}\not=\emptyset$.
Let $z\in \mathcal{A}$. Then there exists $f\in J(x)$ such that
$f(z)=0$. Clearly, the choice of $z$ implies that $f(y)>0$. Thus,
using Lemma~\ref{normderivative}, we arrive at a contradiction.
Hence if $y\in K$ then $y=v_2$ and this proves that
$(-x)\perp_{B^*}y$.

\begin{center}
\begin{tikzpicture}[domain=0:2, scale = 1.40]
\draw[ line width = 0.30mm]   plot[domain=-1:1] (\x,
{0.5-0.5*\x*\x});

\draw[line width = 0.30mm] plot[smooth,domain=-1:1] (\x,
{-0.5+0.5*\x*\x});

\draw[thick] (-1,0)--(1,0) node[right]{};


\draw[thick] plot[smooth,domain=-1:1] (\x, {-\x})node[right]{};

\draw[thick]   plot[smooth,domain=1:-1] ({\x}, {\x})node[right]{};

\draw[thick, dashed]   plot[smooth,domain=1:-1] ({1-\x},
{\x})node[right]{};

\draw[thick, dashed] plot[smooth,domain=-1:1] ({1-\x},
{-\x})node[right]{};

\draw (0.75,0.4)  node {$v_1$};

\draw (-0.75,0.4)  node {$v_2$};

\draw (-0.85,-0.4)  node {$-v_1$};

\draw (0.8,-0.5)  node {$-v_2$};

\draw (1.2,0.0)  node {$x$};

\draw (-1.3,0.0)  node {$-x$};

\draw [fill] (0.41,0.41) circle [radius=.05];

 \draw [fill] (0.41,-0.41) circle [radius=.05];

 \draw [fill] (-0.41,-0.41) circle [radius=.05];

 \draw [fill] (-0.41,0.41) circle [radius=.05];

\draw [fill] (1,0.0) circle [radius=.05];

\draw [fill] (-1,0) circle [radius=.05];

\draw [thick]    (-5,0) circle [radius=1];

\draw[thick] (-6,0)--(-4,0) node[right]{};

\draw[thick] (-5,-1.5)--(-5,1.5) node[above]{};

\draw (-3.8,0.0)  node {$x$};

\draw (-6.3,0.0)  node {$-x$};

\draw (-5.2,1.2)  node {$y$};

\draw (-5.4,-1.2)  node {$-y$};

\draw [fill] (-4,0) circle [radius=.05];

\draw [fill] (-6,0) circle [radius=.05];

\draw [fill] (-5,1) circle [radius=.05];

\draw [fill] (-5,-1) circle [radius=.05];

\draw[thick, dashed] (-4,1.5)--(-4,-1.5) node[right]{};

\draw (-5,-2) node {$Figure~ 1$};

\draw (0,-2)  node {$Figure~ 2$};

\end{tikzpicture}

\end{center}

Let us now prove the sufficient part of $(a)$. Let
$(-x)\perp_{B^*}y$. Observe that by continuity of $\rho_+(x,.)$ in
second variable, it follows that $\rho_+(x,z)=0$ for some $z\in
span\{x,y\}$. Now, by \cite[Theorem 3.2]{CW}, it follows that $z\in
x^{\perp}\cap span\{x,y\}$ and by the homogeneity of Birkhoff-James
orthogonality we may and do assume without loss of generality that
$\|z\|=1$. If $x$ is a smooth point then it follows trivially that
$\{z\in x^{\perp}\cap span\{x,y\}:\|z\|=1\}=\{\pm y\}$. Thus
$\rho_+(x,y)=0$ in this case. Now, we assume that $x$ is a
non-smooth point. It follows from the homogeneity of Birkhoff-James
orthogonality that $x\perp_B y$. If $H=x^{\perp}\cap span\{x,y\}$
then by taking $\epsilon=0$ in \cite[Theorem 2.1]{SPM1}, it follows
that $H=K\cup (-K)$, where $K$ is a normal cone in $span\{x,y\}$.
Let $K$ be generated by $v_1$ and $v_2$, $i.e.$, $K=\{\alpha
v_1+\beta v_2: \alpha\geq0, \beta\geq 0\}$. This situation is
illustrated in figure 2 above. Clearly, $y=v_2$ if $y\in K$ and
$y=-v_1$ if $y\in-K$. We now claim that $z=y$. Let $z\in K$ and
suppose on the contrary that $z\not=y$, $i.e.$, $z\not=v_2$. Then
$z=\frac{tv_2+(1-t)v_1}{\|tv_2+(1-t)v_1\|}$, $t\in[0,1)$. If $f\in
J(x)$ is such that $f(v_2)=0$ then clearly $f(z)>0$. Now, by using
Lemma~\ref{normderivative}, we arrive at a contradiction. Hence
$z=y$ and $\rho_+(x,y)=0$. Similarly, we can prove that $z=y$ when
$z\in-K$.

Proof of $(b)$ follows from similar arguments as above.

\end{proof}

\section{Bounded bilinear operators}

Throughout this section, we will consider normed linear spaces with
dimension strictly greater than one. Our first result shows that if
$\mathcal{T}\in\mathscr{K}(X\times Y,Z)$ is weak-weak continuous
then the norm attainment set of $\mathcal{T}$ is non-empty, where
$X,Y$ are reflexive Banach spaces and $Z$ is a normed linear space.

\begin{lemma}\label{compact w-w continuous}
Let $X,Y$ be reflexive Banach spaces and let $Z$ be a normed linear
space. If $\mathcal{T}\in\mathscr{K}(X\times Y,Z)$ is  weak-weak
continuous then $\mathcal{T}$ attains its norm.
\end{lemma}

\begin{proof}
Without loss of generality we can assume that $\mathcal{T}\not=0$.
Let $\{(x_n,y_n)\}\subseteq S_{X} \times S_{Y}$ be a norming
sequence for $\mathcal{T}$, $i.e.$,
$\|\mathcal{T}(x_n,y_n)\|\longrightarrow \|\mathcal{T}\|$. Now,
reflexivity of $X$ and $Y$ implies that $X\times Y$ is also
reflexive. Thus there exists a weakly convergent subsequence of
$\{(x_n,y_n)\}$, which we again denote by $\{(x_n,y_n)\}$. Let
$(x_0,y_0)\in B_{X \times Y}$ be the weak limit of $\{(x_n,y_n)\}$.
By our assumption, $\mathcal{T}$ is weak-weak continuous, hence
$\{\mathcal{T}(x_n,y_n)\}$ weakly converges to
$\mathcal{T}(x_0,y_0)$.

Now, compactness of $\mathcal{T}$ implies that there exists a convergent
subsequence of $\{\mathcal{T}(x_n,y_n)\}$. Without loss of
generality we again denote this convergent subsequence by
$\{\mathcal{T}(x_n,y_n)\}$. Let $z\in Z$ such that
$\mathcal{T}(x_n,y_n)\longrightarrow z$.

Since $\mathcal{T}(x_0,y_0)$ is weak limit of
$\{\mathcal{T}(x_n,y_n)\}$, we get $z=\mathcal{T}(x_0,y_0)$ and
$\|\mathcal{T}\|=\|\mathcal{T}(x_0,y_0)\|$. $(x_0,y_0)\in B_{X\times
Y}$ implies that $\|x_0\|\leq 1$ and $\|y_0\|\leq 1$. We now claim
that $\|x_0\|=\|y_0\|=1$. If possible assume that $\|x_0\|<1$. Since
$\mathcal{T}$ is not the zero operator, it is easy to observe that
$\|x_0\|>0$ and $\|y_0\|>0$. Now,
$(\frac{x_0}{\|x_0\|},\frac{y_0}{\|y_0\|})\in S_X\times S_Y$ and it
follows from the bilinearity of $\mathcal{T}$ that
$\|\mathcal{T}(\frac{x_0}{\|x_0\|},\frac{y_0}{\|y_0\|})\|=\frac{1}{\|x_0\|\|y_0\|}{\|\mathcal{T}(x_0,y_0)\|}>\|\mathcal{T}\|$,
contradicting the definition of $\|\mathcal{T}\|$. Thus $\|x_0\|=1$
and similar arguments will show that $\|y_0\|=1$. Hence
$(x_0,y_0)\in M_{\mathcal{T}}$ and this completes the proof of the
lemma.
\end{proof}

If $X,Y,Z$ are finite-dimensional Banach spaces then for any
operator $\mathcal{T}\in\mathscr{B}(X\times Y,Z)$, it follows
trivially that $M_{\mathcal{T}}\not=\emptyset$. For
finite-dimensional Banach spaces $X,Y,Z$ and operators $\mathcal{T},
\mathcal{A}\in\mathscr{B}(X\times Y,Z)$, a complete characterization
of $\mathcal{T}\perp_B\mathcal{A},$ involving the elements of
$M_{\mathcal{T}},$ was obtained in \cite[Theorem 2.4]{DS}. This
characterization is an extension of the complete characterization of
Birkhoff-James orthogonality of compact linear operators obtained in
\cite[Theorem 2.1]{SPM} to the bilinear setting, in the
finite-dimensional case. If $X$ is a reflexive Banach space and $Y$
is a normed linear space then for any compact linear operator $T$
from $X$ to $Y,$ the norm attainment set of $T$ is non-empty.
\cite[Theorem 2.1]{SPM} provides a complete characterization of
${T}\perp_B{A}$ involving norm attaining elements of $T$, where $T$,
$A$ are compact linear operators from a reflexive Banach space $X$
to a normed linear space $Y$.

Lemma~\ref{compact w-w continuous} implies that if
$\mathcal{T}\in\mathscr{K}(X\times Y,Z)$ is  weak-weak continuous
then $M_{\mathcal{T}}\not=\emptyset$, where $X,Y$ are reflexive
Banach spaces and $Z$ is a normed linear space. Using this fact, our
next result provides a complete characterization of
$\mathcal{T}\perp_B\mathcal{A},$ where $X,Y$ are reflexive Banach
spaces, $Z$ is a normed linear space and
$\mathcal{T},\mathcal{A}\in\mathscr{K}(X\times Y,Z)$ are weak-weak
continuous.
 We note that the following result extends \cite[Theorem 2.4]{DS} and
\cite[Theorem 2.1]{SPM}.

\begin{theorem}\label{compact operator orthogonality}
Let $X,Y$ be reflexive Banach spaces and let $Z$ be a normed linear
space. Let $\mathcal{T}\in\mathscr{K}(X\times Y,Z)$ be weak-weak
continuous. Then for any weak-weak continuous bilinear operator
$\mathcal{A}\in\mathscr{K}(X\times Y,Z)$, $\mathcal{T}\perp_B
\mathcal{A}$ if and only if there exists $(x,y),(u,v)\in
M_{\mathcal{T}}$ such that $\mathcal{A}(x,y)\in
(\mathcal{T}(x,y))^+$ and $\mathcal{A}(u,v)\in
(\mathcal{T}(u,v))^-$.
\end{theorem}

\begin{proof}
For the sufficient part, observe that if $\mathcal{A}(x,y)\in
(\mathcal{T}(x,y))^+$ for some $(x,y)\in M_{\mathcal{T}}$ then
$\|\mathcal{T}+\lambda \mathcal{A}\|\geq \|\mathcal{T}(x,y)+\lambda
\mathcal{A}(x,y)\|\geq \|\mathcal{T}(x,y)\|=\|\mathcal{T}\|$ for all
$\lambda\geq 0$. Similarly if $\mathcal{A}(u,v)\in
(\mathcal{T}(u,v))^-$ for some $(u,v)\in M_{\mathcal{T}}$ then
$\|\mathcal{T}+\lambda \mathcal{A}\|\geq \|\mathcal{T}\|$ for all
$\lambda\leq 0$.

For the necessary part, observe that for each $n\in\mathbb{N}$,
$\mathcal{T}+\frac{1}{n}\mathcal{A}\in\mathscr{K}(X\times Y,Z)$ and
$\mathcal{T}+\frac{1}{n}\mathcal{A}$ is weak-weak continuous. Thus
from Lemma~\ref{compact w-w continuous}, it follows that
$(x_n,y_n)\in M_{\mathcal{T}+\frac{1}{n}\mathcal{A}}$ for some
$(x_n,y_n)\in S_{X}\times S_{Y}$. Since $X\times Y$ is a reflexive
Banach space, $B_{X\times Y}$ is weakly compact. Without loss of
generality we assume that $(x_n,y_n)$ converges weakly to $(x,y)$.
By using arguments, as given in the proof of Lemma~\ref{compact w-w
continuous}, we can show that $(x,y)\in M_{\mathcal{T}}$. Also, by
using the arguments in the proof of \cite[Theorem 2.1]{SPM}, it
follows that  $\mathcal{A}(x,y)\in(\mathcal{T}(x,y))^+$. Similarly,
considering the compact bilinear and weak-weak continuous operators
$\mathcal{T}-\frac{1}{n}\mathcal{A}$, $n\in\mathbb{N}$, it follows
that there exists some $(u,v)\in M_{\mathcal{T}}$ and
$\mathcal{A}(u,v)\in(\mathcal{T}(u,v))^-$.
\end{proof}

If in addition to the assumptions of the above theorem, we assume
that $M_{\mathcal{T}}=\{(\pm x_0,\pm y_0)\}$  for some $(x_0,y_0)\in
S_{X}\times S_{Y}$ then we obtain the following immediate corollary.

\begin{cor}\label{w-w continuous}
Let $X,Y$ be reflexive Banach spaces and let $Z$ be a normed linear
space. Suppose $\mathcal{T}, \mathcal{A}\in\mathscr{K}(X\times Y,Z)$
are such that $\mathcal{T}, \mathcal{A}$ are weak-weak continuous
and $M_{\mathcal{T}}=\{(\pm x_0, \pm y_0)\}$ for some $(x_0,y_0)\in
S_{X}\times S_{Y}$. Then $\mathcal{T}\perp_B \mathcal{A}$ if and
only if $\mathcal{T}(x_0,y_0)\perp_B \mathcal{A}(x_0,y_0)$.
\end{cor}

A complete characterization of Birkhoff-James orthogonality of
bounded linear operators between normed linear spaces (without
assuming any restriction on the norm attainment set of the operator)
was obtained in \cite[Theorem 2.4]{SPM}. The following result shows
that a similar result is also true for bounded bilinear operators.
We would like to mention that in the bilinear case, the proof
follows from similar arguments as given in \cite[Theorem 2.4]{SPM}.

\begin{theorem}\label{linear operator orthogonality}
Let $X,Y,Z$ be normed linear spaces and let $\mathcal{T}\in
\mathscr{B}(X\times Y,Z)$. Then for any $\mathcal{A}\in
\mathscr{B}(X\times Y,Z)$, $\mathcal{T}\perp_B \mathcal{A}$ if and
only if either of the conditions in $(a)$ or in $(b)$ hold true:

\begin{itemize}
\item[(a)] There exists a sequence  $\{(x_n,y_n)\}$ in $S_X\times
S_Y$ such that $\|\mathcal{T}(x_n,y_n)\|\longrightarrow
\|\mathcal{T}\|$ and $\|\mathcal{A}(x_n,y_n)\|\longrightarrow 0,$ as
$n\longrightarrow \infty$.

\item[(b)] There exist two sequences $\{(x_n,y_n)\}$, $\{(u_n,v_n)\}$  in $S_X\times
S_Y$ and two sequences of positive real numbers $\{\varepsilon_n\}$,
$\{\delta_n\}$ such that

\item[(i)] $\varepsilon_n\longrightarrow 0$,
$\delta_n\longrightarrow 0$ as $n\longrightarrow \infty$.

\item[(ii)] $\|\mathcal{T}(x_n,y_n)\|\longrightarrow
\|\mathcal{T}\|$ and $\|\mathcal{T}(u_n,v_n)\|\longrightarrow
\|\mathcal{T}\|,$ as $n\longrightarrow \infty$.

\item[(iii)] $\mathcal{A}(x_n,y_n)\in
(\mathcal{T}(x_n,y_n))^{+\varepsilon_n}$ and
$\mathcal{A}(u_n,v_n)\in (\mathcal{T}(u_n,v_n))^{-\delta_n}$ for all
$n\in\mathbb{N}$.
\end{itemize}
\end{theorem}

In \cite[Theorem 3.3]{SPMR}, Sain et al. observed that any non-zero
compact linear operator from a reflexive Banach space $X$ to a
normed linear space $Y$ is smooth if and only if $M_{T}=\{\pm x_0\}$
for some $x_0\in S_X$ and $Tx_0$ is a smooth point in $Y$. In
\cite[Theorem 2.5]{DS}, Sain extended this result to smooth bounded
bilinear operator from $X\times Y$  to $Z$, where $X,Y,Z$ are
finite-dimensional Banach spaces. Our next aim is to continue the
study of smooth compact bilinear operators from $X\times Y$ to $Z$
with non-empty norm attainment set, where $X,Y$ are
infinite-dimensional reflexive Banach spaces and $Z$ is any normed
linear space. Observe that for any weak-weak continuous bilinear
operator $\mathcal{T}\in\mathscr{K}(X\times Y, Z),$ it follows from
Lemma~\ref{compact w-w continuous} that
$M_\mathcal{T}\not=\emptyset$, where $X,Y$ are reflexive Banach
spaces. For this reason we will restrict ourselves to  weak-weak
continuous compact bilinear operators. First we prove that if the
norm attainment set for a smooth bilinear operator $\mathcal{T}\in
\mathscr{B}(X\times Y, Z)$ is non-empty then $M_{\mathcal{T}}$ can
have only four points.

\begin{theorem}\label{smooth norm attainment}
Let $X,Y,Z$ be normed linear spaces. If $\mathcal{T}\in
\mathscr{B}(X\times Y, Z)$ is a smooth point and
$M_{\mathcal{T}}\not=\emptyset$ then there exists $(x_0,y_0)\in
S_X\times S_Y$ such that $M_{\mathcal{T}}=\{(\pm x_0,\pm y_0)\}$.
\end{theorem}

\begin{proof}
Suppose on the contrary that  $M_{\mathcal{T}}\not=\{(\pm x_0,\pm
y_0)\}$ for any $(x_0,y_0)\in S_{X}\times S_{Y}$. Since
$M_{\mathcal{T}}\not=\emptyset$, this implies that there exist
$(x_1,y_1), (x_2,y_2)\in M_{\mathcal{T}}$ such that either
$x_1\not=\pm x_2$ or $y_1\not=\pm y_2$. Clearly, it is not possible
that  $x_1$, $x_2$ are linearly dependent in $X$ and $y_1$, $y_2$
are linearly dependent in $Y$. We now consider the
following remaining cases.\\

Case I: We assume that $x_1,x_2$ are linearly independent in $X$. In
this case, we first assume that $y_1,y_2$ are linearly independent
in $Y$. There exist scalars $a_0,b_0$ such that $x_1\perp_B
(a_0x_1+x_2)$ and $y_1\perp_B (b_0y_1+y_2)$. Let $x_2'=a_0x_1+x_2$
and $y_2'=b_0y_1+y_2$. Let $H_1$, $G_1$ be subspaces of
$1$-codimension in $X$, $Y$, respectively, such that $x_2'\in H_1$,
$x_1\perp_B H_1$ and $y_2'\in G_1$, $y_1\perp_B G_1$. Suppose $H_2$,
$G_2$ are subspaces of $1$-codimension in $H_1$, $G_1$,
respectively, such that $x_2'\perp_B H_2$ and $y_2'\perp_B G_2$.
Now, every element $x\in X$, and $y\in Y$ can be written uniquely as
$x=\alpha_1 x_1+\beta_1 x_2'+h$, $y=\alpha_2 y_1+\beta_2 y_2'+g$ for
some scalars $\alpha_1,\alpha_2,\beta_1,\beta_2$ and $h\in H_2$,
$g\in G_2$. If $(x,y)\in S_{X}\times S_{Y}$ then $(x,y)=(\alpha_1
x_1+\beta_1 x_2'+h, \alpha_2 y_1+\beta_2 y_2'+g)$. Also, $\|\alpha_1
x_1\|\leq \|\alpha_1 x_1+\beta_1 x_2'+h\|= 1$ and $1= \|\alpha_1
x_1+\beta_1 x_2'+h\|\geq \|\beta_1 x_2'+h\|-\|\alpha_1 x_1\|\geq
\|\beta_1 x_2'\|-1$. It follows from the linear independence of
$x_1$ and $x_2$ that $\|x_2'\|>0$. Thus $|\alpha_1|=\|\alpha_1
x_1\|\leq 1$ and $|\beta_1|\leq \frac{2}{\|x_2'\|}$. Using similar
arguments, we can show that $|\alpha_2|\leq 1$ and $|\beta_2|\leq
\frac{2}{\|y_2'\|}$. Let $\{h_\alpha: \alpha \in \Lambda_1\}$,
$\{g_\beta: \beta \in \Lambda_2\}$ be a Hamel basis for $H_2$ and
$G_2$, respectively. Then $\{x_1, x_2',h_\alpha: \alpha \in
\Lambda_1\}$, $\{y_1,y_2', g_\beta: \beta \in \Lambda_2\}$ form a
Hamel basis for $X$ and $Y$, respectively. We now define a bilinear
operator $\mathcal{A}$ from $X\times Y$ to $Z$ by:
\begin{align*}
\mathcal{A}(x_1,y_1)=\mathcal{T}(x_1,y_1),
\mathcal{A}(x_1,y_2')=b_0\mathcal{T}(x_1,y_1),
\mathcal{A}(x_1,g_\beta)=0,\\
\mathcal{A}(x_2',y_1)=a_0\mathcal{T}(x_1,y_1),\mathcal{A}(x_2',y_2')=a_0b_0\mathcal{T}(x_1,y_1), \mathcal{A}(x_2',g_\beta)=0,\\
\mathcal{A}(h_\alpha,y_1)=0, \mathcal{A}(h_\alpha,y_2')=0,
\mathcal{A}(h_\alpha,g_\beta)=0
\end{align*}
for all $\alpha \in \Lambda_1$  and  $\beta\in \Lambda_2$.

We claim that $\mathcal{A} \in \mathscr{B}(X\times Y,Z)$. Observe
that if $(x,y)=(\alpha_1 x_1+\beta_1 x_2'+h, \alpha_2 y_1+\beta_2
y_2'+g)\in S_X\times S_Y$ then
\begin{align*}
\mathcal{A}(x,y)=\alpha_1\alpha_2\mathcal{T}(x_1,y_1)+\alpha_1\beta_2b_0\mathcal{T}(x_1,y_1)+
\beta_1\alpha_2a_0\mathcal{T}(x_1,y_1)+\beta_1\beta_2a_0b_0\mathcal{T}(x_1,y_1)
\end{align*}
and thus $\mathcal{A} \in \mathscr{B}(X\times Y,Z)$. Also, it
follows from the definition of $\mathcal{A}$ that
$\mathcal{A}(x_2,y_2)=0$.

We now define $\mathcal{B}\in\mathscr{B}(X\times Y, Z)$ by
$\mathcal{B}=\mathcal{T}-\mathcal{A}$. Observe that
$\mathcal{B}(x_1,y_1)=0$. Thus
\begin{align*}
\mathcal{T}(x_2,y_2)\perp_B \mathcal{A}(x_2,y_2),
\mathcal{T}(x_1,y_1)\perp_B \mathcal{B}(x_1,y_1)
\end{align*}
and hence $\mathcal{T}\perp_B \mathcal{A}$ and $\mathcal{T}\perp_B
\mathcal{B}$. Now, by the smoothness of $\mathcal{T}$ we have
$\mathcal{T}\perp_B(\mathcal{A}+\mathcal{B})$. This leads to a
contradiction as $\mathcal{A}+\mathcal{B}=\mathcal{T}$. Thus in this
case we get $M_{\mathcal{T}}=\{(\pm x_0,\pm y_0)\}$ for some $(\pm
x_0,\pm
y_0)\in S _X\times S_Y$.\\

Now, we assume that $y_1,y_2$ are linearly dependent in $Y$. Then
either $y_1=y_2$ or $y_1=-y_2$. Suppose $G_1$ be a subspace of
$1$-codimension in $Y$ such $y_1\perp_B G_1$. If $\{g_\beta:\beta\in
\Lambda_2\}$ is a Hamel basis for $G_1$ then $\{y_1,g_\beta:\beta\in
\Lambda_2\}$ forms a Hamel basis for $Y$. Using the same notations
as in the previous case, let $\{x_1,x_2', h_\alpha: \alpha \in
\Lambda_1\}$ be a Hamel basis for $X$. We now define a bilinear
operator $\mathcal{A}$ from $X\times Y$ to $Z$ by:
\begin{align*}
\mathcal{A}(x_1,y_1)=\mathcal{T}(x_1,y_1),
\mathcal{A}(x_1,g_\beta)=0,\mathcal{A}(x_2',
y_1)=a_0\mathcal{T}(x_1,y_1),\\
\mathcal{A}(x_2',g_\beta)=0, \mathcal{A}(h_\alpha,y_1)=0,
\mathcal{A}(h_\alpha,g_\beta)=0
\end{align*}
for all $\alpha \in \Lambda_1$ and $\beta\in \Lambda_2$. Using the
arguments similar to the previous case we can show that $\mathcal{A}
\in \mathscr{B}(X\times Y,Z)$. We now define
$\mathcal{B}\in\mathscr{B}(X\times Y, Z)$ by
$\mathcal{B}=\mathcal{T}-\mathcal{A}$. Observe that
$\mathcal{A}(x_2,y_2)=0$ and $\mathcal{B}(x_1,y_1)=0$. Thus in this
case also we obtain $\mathcal{T}\perp_B(\mathcal{A}+\mathcal{B})$.
This leads to a contradiction as
$\mathcal{T}=\mathcal{A}+\mathcal{B}$. Thus $M_{\mathcal{T}}=\{(\pm
x_0,\pm y_0)\}$ for some $(\pm x_0,\pm y_0)\in S _X\times S_Y$.\\

Case II: We are only left with the case that $x_1,x_2$ are linearly
dependent in $X$ and $y_1,y_2$ are linearly independent in $Y$. In
this case also, using the arguments as in case I, it follows that
$M_{\mathcal{T}}=\{(\pm x_0,\pm y_0)\}$ for some $(\pm x_0,\pm
y_0)\in S _X\times S_Y$.

This completes the proof of the theorem.
\end{proof}

If in addition to the conditions of Lemma~\ref{compact w-w
continuous}, we assume $\mathcal{T}$ to be smooth then we
immediately obtain  the following result on the norm attainment set
of $\mathcal{T}$.
\begin{cor}
Let $X,Y$ be reflexive Banach spaces and let $Z$ be a normed linear
space. Suppose $\mathcal{T}\in\mathscr{K}(X\times Y, Z)$ is a smooth
and weak-weak continuous bilinear operator. Then
$M_{\mathcal{T}}(x_0,y_0)=\{(\pm x_0,\pm y_0)\}$ for some $(x_0,
y_0)\in S_X\times S_Y$.
\end{cor}

We now prove the analogous result of \cite[Theorem 3.3]{SPMR} and
\cite[Theorem 2.5]{DS} for weak-weak continuous compact bilinear
operators from $X\times Y$ to $Z$, where $X,Y$ are reflexive Banach
spaces and $Z$ is a normed linear space.

\begin{theorem}
Let $X,Y$ be reflexive Banach spaces and let $Z$ be a normed linear
space. Let $\mathcal{T}\in\mathscr{K}(X\times Y, Z)$ be weak-weak
continuous. Then $\mathcal{T}$ is a smooth point in
$\mathscr{B}(X\times Y, Z)$ if and only if there exists some
$(x_0,y_0)\in S_{X}\times S_{Y}$ such that $M_{\mathcal{T}}=\{(\pm
x_0,\pm y_0)\}$ and $\mathcal{T}(x_0,y_0)$ is a smooth point in $Z$.
\end{theorem}

\begin{proof}
Let us first prove the sufficient part of the theorem. Let
$\mathcal{A}_1,\mathcal{A}_2\in\mathscr{B}(X\times Y, Z)$ be such that
$\mathcal{T}\perp_B\mathcal{A}_1$ and
$\mathcal{T}\perp_B\mathcal{A}_2$. It follows from
Corollary~\ref{w-w continuous} that $\mathcal{T}(x_0,y_0)\perp_B
\mathcal{A}_1(x_0,y_0)$ and $\mathcal{T}(x_0,y_0)\perp_B
\mathcal{A}_2(x_0,y_0)$. Now, from the smoothness of
$\mathcal{T}(x_0,y_0),$ we obtain that $\mathcal{T}(x_0,y_0)\perp_B
(\mathcal{A}_1(x_0,y_0)+\mathcal{A}_2(x_0,y_0))$. Since
$(x_0,y_0)\in M_{\mathcal{T}}$, this implies that
$\mathcal{T}\perp_B (\mathcal{A}_1+\mathcal{A}_2)$.

Let us now prove the necessary part of the theorem. It follows from
Lemma~\ref{compact w-w continuous} and Theorem~\ref{smooth norm
attainment} that $M_{\mathcal{T}}=\{(\pm x_0,\pm y_0)\}$ for some
$(x_0,y_0)\in S_{X} \times S_{Y}$. Suppose on the contrary that
$\mathcal{T}(x_0,y_0)$ is not a smooth point in $Z$. Then there
exist $z_1,z_2\in Z$ such that $\mathcal{T}(x_0,y_0)\perp_B z_1$,
$\mathcal{T}(x_0,y_0)\perp_B z_2$ and
$\mathcal{T}(x_0,y_0){\not\perp}_B (z_1+z_2)$. Suppose $H_1$, $G_1$
are subspaces of $1$-codimension in $X$, $Y$, respectively such that
$x_0\perp_B H_1$ and $y_0\perp_B G_1$. Let $\{h_\alpha:\alpha\in
\Lambda_1\}$, $\{g_\beta:\beta\in \Lambda_2\}$ be a Hamel basis for
$H_1$ and $G_1$, respectively. Then $\{x_0,h_\alpha:\alpha\in
\Lambda_1\}$,  $\{y_0,g_\beta:\beta\in \Lambda_2\}$ form a Hamel
basis for $X$ and $Y$, respectively. Now, we define $\mathcal{A},
\mathcal{B}\in\mathscr{K}(X\times Y, Z)$ by
\begin{align*}
\mathcal{A}(x_0,y_0)=z_1, \mathcal{A}(x_0,g_\beta)=0,
\mathcal{A}(h_\alpha, y_0)=0,\mathcal{A}(h_\alpha,g_\beta)=0,\\
\mathcal{B}(x_0,y_0)=z_2, \mathcal{B}(x_0,g_\beta)=0,
\mathcal{B}(h_\alpha, y_0)=0, \mathcal{B}(h_\alpha,g_\beta)=0
\end{align*}
for all $\alpha \in\Lambda_1$, $\beta \in\Lambda_2$. From the
definition of $\mathcal{A}$ and $\mathcal{B},$ it follows that
$\mathcal{T}\perp_B \mathcal{A}$ and $\mathcal{T}\perp_B
\mathcal{B}$. Now, from the smoothness of $\mathcal{T}$ we get that
$\mathcal{T}\perp_B (\mathcal{A}+\mathcal{B})$. Observe that
$\mathcal{A}$  and $\mathcal{B}$ are also weak-weak continuous. Thus
it follows from Theorem~\ref{compact operator orthogonality} that
\begin{align*}
\mathcal{T}(x_0,y_0)\perp_B
(\mathcal{A}(x_0,y_0)+\mathcal{B}(x_0,y_0)).
\end{align*}
This contradicts that $\mathcal{T}(x_0,y_0){\not\perp}_B (z_1+z_2)$.
Hence $\mathcal{T}(x_0,y_0)$ is a smooth point in $Z$.
\end{proof}

A complete characterization of smoothness of bounded linear
operators between normed linear spaces (without assuming any
restriction on the norm attainment set of the operator) in terms of
operator Birkhoff-James orthogonality and s.i.p. was obtained in
\cite[Theorem 2.1]{SPMR}. We observe that the same characterization
extends to the setting of bounded bilinear operators. For the
bilinear case, the proof follows from similar arguments, as given in
\cite[Theorem 2.1]{SPMR}. Therefore, we omit the proof of the
following result.

\begin{theorem}
Let $X,Y,Z$ be normed linear spaces and let $\mathcal{T}\in
\mathscr{B}(X\times Y,Z)$ be non-zero. Then the following conditions
are equivalent:

\begin{itemize}
\item[(a)] $T$ is smooth.

\item[(b)] For any $\mathcal{A}\in \mathscr{B}(X\times Y,Z)$,
$\mathcal{T}\perp_B \mathcal{A}$ if and only if for any norming
sequence $\{(x_n,y_n)\}$ for $\mathcal{T}$, any subsequential limit
of $\{[Ax_n, Tx_n]\}$ is $0$, for any s.i.p. $[~,~]$ on $Z$.
\end{itemize}
\end{theorem}

We end this article with the following two results, giving complete
characterizations of approximate Birkhoff-James orthogonality for
weak-weak continuous compact bilinear operators and bounded bilinear
operators. These results extend \cite[Theorem 3.3]{SPM3} and
\cite[Theorem 3.4]{SPM3}. Proofs of the following characterizations
follow from similar arguments as given in \cite{SPM3}, and are
therefore omitted. We also note that the following results extend
Theorem~\ref{compact operator orthogonality} and Theorem~\ref{linear
operator orthogonality}.

\begin{theorem}
Let $X,Y$ be reflexive Banach spaces and let $Z$ be a Banach space.
Let $\mathcal{T} \in \mathscr{K}(X\times Y,Z)$ be weak-weak
continuous. Then for any weak-weak continuous bilinear operator
$\mathcal{A}\in\mathscr{K}(X\times Y,Z)$,
$\mathcal{T}\perp_B^\epsilon \mathcal{A}$ if and only if there exist
$(x_1,y_1)$, $(x_2,y_2)\in M_\mathcal{T}$ such that
$\|\mathcal{T}(x_1,y_1) +\lambda \mathcal{A}(x_1,y_1)\|^2\geq
\|\mathcal{T}\|^2-2\epsilon\lambda \|\mathcal{T}\|\|\mathcal{A}\|$
for all $\lambda\geq 0$ and $\|\mathcal{T}(x_2,y_2) +\lambda
\mathcal{A}(x_2,y_2)\|^2\geq \|\mathcal{T}\|^2-2\epsilon\lambda
\|\mathcal{T}\|\|\mathcal{A}\|$ for all $\lambda\leq 0$.
\end{theorem}

\begin{theorem}
Let $X,Y,Z$ be normed linear spaces and let $\mathcal{T}\in
\mathscr{B}(X\times Y,Z)$.  Then for any $\mathcal{A}\in
\mathscr{B}(X\times Y,Z)$, $\mathcal{T}\perp_B^\epsilon \mathcal{A}$
if and only if either of the conditions in $(a)$ or in $(b)$  hold
true:

\begin{itemize}
\item[(a)] There exists a sequence  $\{(x_n,y_n)\}$ in $S_{X}\times
S_{Y}$ such that $\|\mathcal{T}(x_n,y_n)\|\longrightarrow
\|\mathcal{T}\|$ and $\|\mathcal{A}(x_n,y_n)\|\longrightarrow
\epsilon \|\mathcal{A}\|$ as $n\longrightarrow \infty$.

\item[(b)] There exists two sequences $\{(x_n,y_n)\}$, $\{(u_n,v_n)\}$  in $S_{X}\times
S_{Y}$ and two sequences of positive real numbers
$\{\varepsilon_n\}$, $\{\delta_n\}$ such that

\item[(i)] $\varepsilon_n\longrightarrow 0$,
$\delta_n\longrightarrow 0$,
$\|\mathcal{T}(x_n,y_n)\|\longrightarrow \|\mathcal{T}\|$ and
$\|\mathcal{T}(u_n,v_n)\|\longrightarrow \|\mathcal{T}\|$ as
$n\longrightarrow \infty$.

\item[(ii)] $\|\mathcal{T}(x_n,y_n)+\lambda \mathcal{A}(x_n,y_n)
\|^2\geq(1-\varepsilon_n^2) \|\mathcal{T}(x_n,y_n)\|^2-2\epsilon
\sqrt{1-\varepsilon_n^2}\|\mathcal{T}(x_n,y_n)\|\|\lambda\mathcal{A}\|$
for all $\lambda\geq 0$.

\item[(iii)] $\|\mathcal{T}(u_n,v_n)+\lambda \mathcal{A}(u_n,v_n)
\|^2\geq(1-\delta_n^2) \|\mathcal{T}(u_n,v_n)\|^2-2\epsilon
\sqrt{1-\delta_n^2}\|\mathcal{T}(u_n,v_n)\|\|\lambda\mathcal{A}\|$
for all $\lambda\leq 0$.
\end{itemize}
\end{theorem}

\bibliographystyle{amsplain}

\end{document}